\documentclass{amsart}
\usepackage[utf8]{inputenc}
\usepackage[utf8]{inputenc}
\usepackage[T1]{fontenc}
\usepackage[all]{xy}
\usepackage{lipsum}
\usepackage{url}
\usepackage{stackrel}

\usepackage{amsmath,amsthm,amssymb}

\usepackage{graphicx}

 \newtheorem{thm}{Theorem}[section]
 
 \newtheorem{cor}[thm]{Corollary}
 \newtheorem{lem}[thm]{Lemma}
 \newtheorem{prop}[thm]{Proposition}
 \theoremstyle{definition}
 
 \theoremstyle{remark}
 \newtheorem{rem}[thm]{Remark}
 \numberwithin{equation}{section}

\begin{document}

\title[CATEGORY OF THE CONFIGURATION SPACE]
{LUSTERNIK-SCHNIRELMANN CATEGORY OF THE CONFIGURATION SPACE OF COMPLEX PROJECTIVE SPACE }

\author{Cesar A. Ipanaque Zapata }
\date{} 
\address[Cesar A. Ipanaque Zapata]{Deparatmento de Matem\'{a}tica,UNIVERSIDADE DE S\~{A}O PAULO
INSTITUTO DE CI\^{E}NCIAS MATEM\'{A}TICAS E DE COMPUTA\c{C}\~{A}O -
USP , Avenida Trabalhador S\~{a}o-carlense, 400 - Centro CEP:
13566-590 - S\~{a}o Carlos - SP, Brasil}
\email{cesarzapata@usp.br}

\maketitle

\begin{abstract}

The Lusternik-Schnirelmann category $cat(X)$ is a homotopy invariant which is a numerical bound on the number of critical points of a smooth function on a manifold. Another similar invariant is the topological complexity $TC(X)$ (a la Farber) which has interesting applications in Robotics, specifically, in the robot motion planning problem. In this paper we calculate the Lusternik-Schnirelmann category and as a consequence we calculate the topological complexity of the two-point ordered configuration space of $\mathbb{CP}^n$ for every $n\geq 1$. 
    
\end{abstract}


\section{INTRODUCTION}

The ordered configuration space of $k$ distinct points of a topological space $X$ (see \cite{fadell1962configuration}) is the subset 
 \[F(X,k)=\{(x_1,\cdots,x_k)\in X^k\mid ~~x_i\neq x_j\text{ for all } i\neq j \}\] topologised, as a subspace of the Cartesian power $X^k$. This space has been used in robotics when one controls multiple objects simultaneously, trying to avoid collisions between them \cite{farber2008invitation}. 

The first definition of category was given by Lusternik and Schnirelmann \cite{lusternik}. Their definition was a consequence of an investigation to obtain numerical bounds for the number of critical points of a smooth function on a manifold.

Here we follow a definition of category, one greater than category given in  \cite{cornea2003lusternik}. We say that the Lusternik-Schnirelmann category or category of a topological space $X$, denoted $cat(X)$, is the least integer $m$ such that $X$ can be covered with $m$ open sets, which are all contractible within $X$.  One of the basic properties of $cat(X)$
is its homotopy invariance (\cite{cornea2003lusternik}, Theorem 1.30).

Proposition \ref{prop 1} belows gives the general lower and upper bound of the category of a space $X$:


\begin{prop}\label{prop 1}
\begin{enumerate}
    \item[(1)] (\cite{roth2008category}, Section 2, Proposition 2.1-5), pg. 451) If $X$ is an $(n-1)-$connected CW-complex, then \[cat(X)\leq \frac{dim(X)}{n}+1.\] 
    \item[(2)] Let $R$ be a commutative ring with unit and $X$ be a space. We have \[1+cup_R(X)\leq cat(X)\] where  $cup_R(X)$ is the least integer $n$ such that all $(n+1)-$fold cup products vanish in the reduced cohomology $\widetilde{H^\star}(X;R)$ (\cite{cornea2003lusternik}, Theorem 1.5, pg. 2).
\end{enumerate}
\end{prop}

On the other hand we recall the definition of topological complexity (see \cite{farber2003topological} for more details).  The \textit{Topological complexity} of a path-connected space $X$ is the least integer $m$ such that the Cartesian product $X\times X$ can be covered with $m$ open subsets $U_i$, \begin{equation*}
        X \times X = U_1 \cup U_2 \cup\cdots \cup U_m 
    \end{equation*} such that for any $i = 1, 2, \ldots , m$ there exists a continuous function $s_i : U_i \longrightarrow PX$, $\pi\circ s_i = id$ over $U_i$. If no such $m$ exists we will set $TC(X)=\infty$. Where $PX$ denote the space of all continuous paths $\gamma: [0,1] \longrightarrow X$ in $X$ and  $\pi: PX \longrightarrow X \times X$ denotes the
map associating to any path $\gamma\in PX$ the pair of its initial and end points $\pi(\gamma)=(\gamma(0),\gamma(1))$. Equip the path space $PX$ with the compact-open topology. 

The central motivating result of this paper is the Lusternik-Schnirelmann category of the configuration space of $2$ distinct points in Complex Projective $n-$space for all $n\geq 1$,

 \begin{thm}\label{theor}  For $n\geq 1$,
\[cat(F(\mathbb{CP}^n,2))=2n.\]
 \end{thm}
 
As an application we have the following statement.
 
 \begin{cor}\label{coro}
  For $n\geq 1$,
\[TC(F(\mathbb{CP}^n,2))=4n-1.\]
 \end{cor}

 \section{PROOF}
 In this section we proof Theorem \ref{theor} and Corollary \ref{coro}. We begin by proving two lemmas needed for our proofs.
 
 \begin{lem} 
For $n\geq 1$, \[H_q(F(\mathbb{CP}^n,2);\mathbb{Z})=\left\{
  \begin{array}{ll}
    \mathbb{Z}^{\oplus (\frac{q}{2}+1)}, & \hbox{$q=0,2,4,\cdots,2(n-1)$;} \\
    \mathbb{Z}^{\oplus (2n-\frac{q}{2})} , & \hbox{$q=2n,2n+2,2n+4,\cdots,2n+2(n-1)$;} \\
    0, & \hbox{otherwise.}
  \end{array}
\right.
\] 
 \end{lem}
 \begin{proof}
   By the Leray-Serre espectral sequence (\cite{mccleary2001user}, Theorem 5.4, pg. 139) of the fibration $F(\mathbb{CP}^n,2)\longrightarrow \mathbb{CP}^n,~(x,y)\mapsto x$ with fibre $\mathbb{CP}^{n-1}$ (\cite{fadell1962configuration}, Theorem 1), we have the $E^2-$term\[E^2_{p,q}=H_p(\mathbb{CP}^n;\mathbb{Z})\otimes H_q(\mathbb{CP}^{n-1};\mathbb{Z}) \] and all those differentials are zero (see Figure \ref{seq}). So this Lemma follows.
    \end{proof}
    
    \begin{figure}[!h]
 \caption{$E^2-$term.}
 \label{seq}
 \centering
 \includegraphics[scale=0.5]{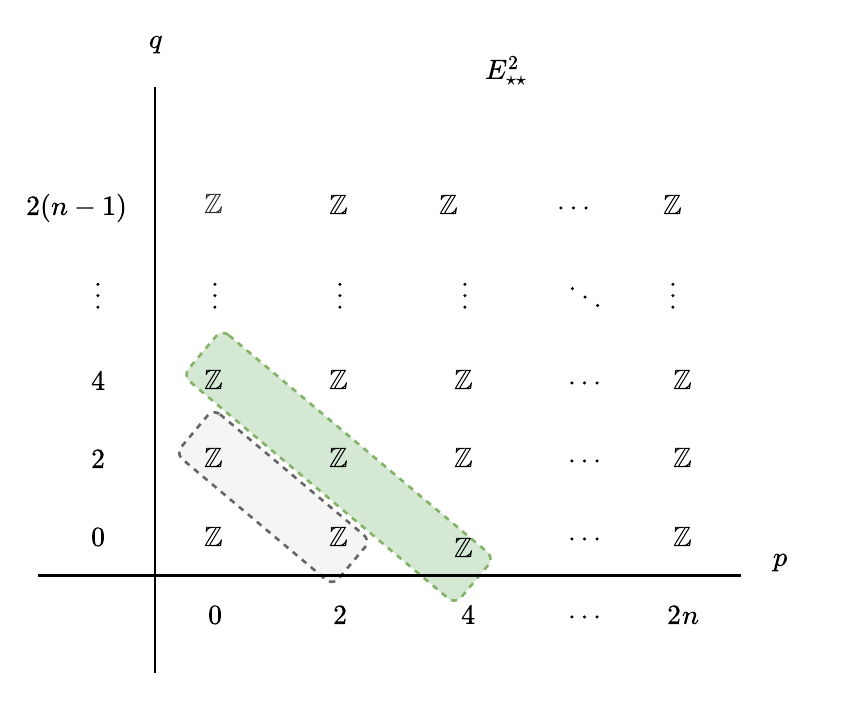}
\end{figure} 
 
We recall that $F(\mathbb{CP}^n,2)$ is simply-connected, since $\mathbb{CP}^n$ and $\mathbb{CP}^{n-1}$ are. By (\cite{hatcher2002algebraic}, Proposition 4C.1) we have:
 
 \begin{cor}\label{coro2}
The configuration space of complex projective space $F(\mathbb{CP}^n,2)$ has the homotopy type of a CW complex which has $j+1$\quad $2j-$cells ($j=0,1,\ldots ,n-1$),  and $2n-j$\quad $2j-$cells ($j=n,n+1,n+2,\ldots,n+(n-1)$). In particular,  $F(\mathbb{CP}^n,2)$ has the homotopy type of a $2(2n-1)-$dimensional finite  CW complex. 
 \end{cor}

The multiplicative structure of the cohomological algebra of the configuration space $F(\mathbb{C}P^n,2)$ was given by (\cite{sohail2010cohomology}, Theorem 2, pg. 412):
  
  \[H^\star(F(\mathbb{CP}^n,2);\mathbb{C})=\dfrac{\mathbb{C}[a_1,a_2]}{\langle r_n(a_1,a_2);a_1^{n+1};a_2^{n+1}\rangle},\] where $deg (a_1)=deg (a_2)=2$ and $r_n(x,y)=x^n+x^{n-1}y+\cdots +y^n$. Thus, we can conclude $a_1^na_2^n=0$ and $a_1^{n-1}a_2^n\neq 0$, since $a_1^na_2^n=r_n(a_1,a_2)a_2^n=0$ and $a_1^{n-1}a_2^n$ is a unique (up to sign) generator of $H^{4n-2}=\mathbb{C}$.

\begin{lem}\label{lem3}
\[cup_{\mathbb{C}}(F(\mathbb{CP}^n,2))=2n-1.\]
\end{lem}
\begin{proof}
 We just have to note that $a_1^na_2^n=0$ and $a_1^{n-1}a_2^n\neq 0$.  
\end{proof}

\noindent \textit{Proof of Theorem \ref{theor}.}

It follows using Corrollary \ref{coro2},  Lemma \ref{lem3} and Proposition \ref{prop 1}.
 
\begin{flushright}
 $\square$
 \end{flushright}

\noindent \textit{Proof of corollary \ref{coro}.}   
 
 Since $F(\mathbb{CP}^n,2)$ is path-connected and paracompact, the inequality \[TC(F(\mathbb{CP}^n,2))\leq 4n-1\] follows from Theorem \ref{theor} and (\cite{farber2003topological}, Section 3, Theorem 5, pg. 215). On the other hand, $1\otimes a_1-a_1\otimes 1$ and $1\otimes a_2-a_2\otimes 1\in H^\star(F(\mathbb{CP}^n,2);\mathbb{C})\otimes H^\star(F(\mathbb{CP}^n,2);\mathbb{C})$ are zero-divisors whose $(2n-1)-$th power
 \begin{eqnarray*}
 (1\otimes a_1-a_1\otimes 1)^{2n-1} &=& pa_1^{n-1}\otimes a_1^n+qa_1^{n}\otimes a_1^{n-1}; \\
 (1\otimes a_2-a_2\otimes 1)^{2n-1} &=& pa_2^{n-1}\otimes a_2^n+qa_2^{n}\otimes a_2^{n-1},
 \end{eqnarray*} where $p=(-1)^{n-1}{2n-1 \choose n-1} $ and $q=(-1)^{n}{2n-1 \choose n} $.
 
 Thus, we have
 \begin{equation*}
     (1\otimes a_1-a_1\otimes 1)^{2n-1}(1\otimes a_2-a_2\otimes 1)^{2n-1} = 2p^2 a_1^{n-1}a_2^n\otimes a_1^{n-1}a_2^n
 \end{equation*} does not vanish. The opposite inequality \[TC(F(\mathbb{CP}^n,2))\geq 4n-1\] now follows from (\cite{farber2003topological}, Theorem 7).
 
\begin{flushright}
 $\square$
 \end{flushright}
 
\begin{rem}
 Corollary \ref{coro} in the case $n=1$ also was calculated by Michael Farber and Daniel Cohen in (\cite{cohen2011topological}, Theorem A).
 \end{rem}
 
 \begin{rem}
 Theorem \ref{theor} shows that the configuration space $F(\mathbb{CP}^n,2)$ satisfies the Ganea's conjecture, because $cat(F(\mathbb{CP}^n,2))=cup_{\mathbb{C}}(F(\mathbb{CP}^n,2))+1$. 
 \end{rem}
 
 \begin{rem}
  By (\cite{farber2003topologicalproj}, Corollary 3.2) we have \begin{equation}
      TC(M)=dim(M)+1\label{symplectic}
  \end{equation}  when $M$ is a closed simply connected symplectic manifold. Corollary \ref{coro} shows that the analogous statement (\ref{symplectic}) for non compact cases does not hold.
 \end{rem}
 
 \begin{rem}
 We will compare the result stated in Corollary \ref{coro} with the topological complexity of the Cartesian product $\mathbb{CP}^n\times \mathbb{CP}^n$. By (\cite{farber2003topologicalproj}, Corollary 3.2) we have \begin{equation*}
 TC(\mathbb{CP}^n\times \mathbb{CP}^n) = 4n+1.
 \end{equation*} Thus, on the complex projective space $\mathbb{CP}^n$, the complexity of the collision-free motion planning problem for $2$ robots is \textit{less} complicated than the complexity of the similar problem when the robots are allowed to collide.  This example also\footnote{This phenomenon occurs also on a surface of high genus (see \cite{cohen2011topological}).} provides an illustration of the fact that the concept $TC(X)$ reflects only the \textit{topological} complexity, which is just a part of the \textit{total} complexity of the problem.
\end{rem}

\begin{rem}
 We note that the configuration space $F(\mathbb{CP}^n,2)$ is the space of all lines in the complex projective space $\mathbb{CP}^n$, since two points in $\mathbb{CP}^n$ generate a subspace of dimension 1. More general, \cite{berceanu2012braid} the ordered configuration space $F(\mathbb{CP}^n,k)$ has a stratification with complex submanifolds as follows:
\[F(\mathbb{CP}^n,k)=\coprod_{i=1}^{n}F^i(\mathbb{CP}^n,k),\]
where $F^i(\mathbb{CP}^n,k)$ is the ordered configuration space of all $k$ points in $\mathbb{CP}^n$ generating a subspace of dimension $i$.
\end{rem}

\begin{rem}
There is no discussion of what might happen for more than two points. Thus, it is interesting to calculate the TC for the ordered configuration space $F(\mathbb{CP}^n,k)$  when $k\geq 3$. In general,  calculate the TC for the ordered configuration space $F(V,k)$ where $V$ is a smooth complex projective variety.
\end{rem}

\noindent \textbf{Acknowledgments} I am very grateful to Jesús González and my advisor Denise de Mattos for their comments and encouraging remarks which were of invaluable mental support. Also, the author wishes to acknowledge support for this research, from FAPESP 2016/18714-8.

\end{document}